\documentclass{amsart}
\usepackage{mathrsfs,amssymb,setspace}
\usepackage{verbatim, dsfont, enumerate, color, stmaryrd, bbold, bm, amsmath,mathabx}

\theoremstyle{plain}
\newtheorem{theorem}{Theorem}[section]
\newtheorem{lemma}[theorem]{Lemma}
\newtheorem{proposition}[theorem]{Proposition}
\newtheorem{corollary}[theorem]{Corollary}

\theoremstyle{definition}
\newtheorem{notation}[theorem]{Notation}

\newtheorem{definition}[theorem]{Definition}

\theoremstyle{remark}
\newtheorem{remark}[theorem]{Remark}

\usepackage[all]{xy}

\newcommand{\ldbr}{\{\!\!\{}
\newcommand{\rdbr}{\}\!\!\}}

\newcommand{\cref}[1]{Corollary \ref{#1}}

\begin{document}


\title[A Class of Non-Contracting Weakly Branch Groups]{A Branch Group in a Class of Non-Contracting Weakly Regular Branch Groups}
\author[Sagar Saha]{Sagar Saha}
\address{Department of Mathematics,  Indian Institute of Technology Guwahati, Guwahati, India}
\email{sagarsaha@iitg.ac.in}
\author[K. V. Krishna]{K. V. Krishna}
\address{Department of Mathematics, Indian Institute of Technology Guwahati, Guwahati, India}
\email{kvk@iitg.ac.in}

\begin{abstract}
We provide a class of non-contracting groups containing an infinite family of fractal and weakly regular branch groups, and study certain properties including abelianization, just infiniteness, and word problem. We present an example of a branch group in this class and show that it is of exponential growth. It seems this is the first example of a non-contracting branch group constructed explicitly.   
\end{abstract}



\maketitle

\vspace{.5cm}

\section{Introduction}

The exploration of groups acting on regular rooted trees gained significant importance since the pioneering work of Grigorchuk in 1980's. Notably, a simple counterexample to the General Burnside Problem was provided through the Grigorchuk group, which is an automorphism group of a 2-regular rooted tree \cite{grigorchuk1980burnside}. The Grigorchuk group also unveiled interesting properties, e.g., it is amenable but not elementary amenable, just infinite, and importantly, it is the first example of a group of intermediate word growth \cite{Grigorchuk1984}. Subsequently, many generalizations and different examples of groups acting on regular rooted trees were studied. For instance, in \cite{gupta1983burnside}, Gupta and Sidki introduced a family of $p$-groups for each odd prime $p$, and in \cite{bartholdi2001word}, Bartholdi and Sunic studied a family of groups generalizing the Grigorchuk group.

Following the emergence of similar examples that share common features, there was an attempt to categorize these instances into a cohesive framework which enables a more organized approach to their studies. This endeavor resulted in the classification of these instances into self-similar groups, and (weakly) branch groups. The groups falling in the intersection of these classes received special attention in the literature as they have many tools enabling to study properties like just infiniteness, fractalness, maximal subgroups, congruence subgroup problems, amenability, L-presentations (cf.   \cite{Bartholdi2003presentation,Bou-Rabee2020,Noce2020engel,Francoeur2020,Noce2021Hausdorff, Skipper2020}). The aforesaid examples are in the intersection of these classes. In addition, many other examples were extensively studied in the literature (e.g., \cite{Bartholdi2010,grigorchuk2002torsion}).  
Almost all the examples have a specific property, which defines contracting groups.

In contrast to contracting groups, the literature has very limited examples of non-contracting weakly branch groups. The groups constructed explicitly in \cite{Dahmani2005} by Dahmani and in \cite{mamaghani2011fractal} by Mamaghani were only two examples available until recently. While the Hanoi tower group $\Gamma_3$ is a contracting regular branch group, its generalization $\Gamma_d$ (defined based on the game on d pegs, for $d \geq 4$) is a family of non-contracting weakly branch groups (cf. \cite{Grigorchuk2006hanoi,Thesis_Skipper}). For any $d \ge 4$, it is not known whether or not $\Gamma_d$ is a branch group. In \cite{Noce2021}, Noce constructed an infinity family of non-contracting weakly branch groups acting on a $d$-regular rooted tree, for $d \geq 2$.  Due to their non-contracting nature, not much on any of the aforesaid groups is available in the literature. 

In this work, we provide yet another class of non-contracting weakly regular branch groups aiming to establish a branch group in the class. In Section \ref{gen_class}, we present a class of non-contracting self-similar groups acting on $d$-regular rooted trees, for $d \geq 3$, and study certain properties of their elements. In Section \ref{odd_class}, we consider the class of groups when $d$ is odd, and prove that each group in the class is fractal, weakly branch over its commutator subgroup. Further, we establish that the abelianization of each group in this class is isomorphic to $\mathds{Z}^d$, and hence the groups are not just infinite. We also provide an efficient algorithm for the word problem of this class. For $d = 3$, in Section \ref{grp_g3}, we show that the semigroup generated by $A$ (the defining generating set of the group) is free and accordingly observe that the group is of exponential growth. Moreover, we study the structure of its rigid stabilizer and show that the group is a branch group.

\section{Preliminaries}

We now recall the fundamental notions and fix the notations used in this work. We also present a few required properties from the literature. For more details, one may refer to \cite{Bartholdi2003,Grigorchuk2005,Nekrashevych2005}.  

For an integer $d \geq 2$, consider the set $X = \{1, \ldots, d\}$ of first $d$ positive integers. Let $X^*$ be the free monoid of the words over $X$ with respect to concatenation and the empty word be denoted by $\varepsilon$. The $d$-regular rooted tree over $X$, denoted by $T_X$, is a graph with the vertex set $X^*$, rooted at $\varepsilon$, and two vertices $u, v \in X^*$ are adjacent if and only if $u = vx$ for some $x \in X$. In what follows, $T_X$ is simply denoted by $T$. The length of a word $w$ over any set is denoted by $|w|$, which is the number of symbols from the set appearing in $w$ including their repetitions.  Note that the set $X^k$ of all words of length $k$ are the vertices at the level $k$ of $T$. For $u \in X^*$, we write $T_u$ to denote the subtree of $T$ rooted at $u$ consisting of all vertices $v$ such that $u$ is a prefix of $v$, i.e., $v = uu'$ for some $u' \in X^*$. Consider the set $\text{Aut}(T)$ of all graph automorphisms of $T$ preserving the root. Note that $\text{Aut}(T)$ is a group with respect to composition of maps, $gh(u) = h(g(u))$ for $g, h \in \text{Aut}(T)$ and $u \in X^*$. We shall denote the identity element of $\text{Aut}(T)$ by $e$. For $g \in \text{Aut}(T)$, note that $g$ is length preserving, and $g(X^k) = X^k$ for all $k$.

Let $u \in X^*$ and $g \in \text{Aut}(T)$. Since $g$ is length preserving automorphism, for each $v \in X^*$, there exists unique $v' \in X^*$ such that $g(uv) = g(u)v'$. Define $g|_u :T \to T$ by  $g|_u(v) = v'$, called the section of $g$ at $u$. Note that $g|_u \in \text{Aut}(T)$. The sections satisfy the following properties: For any $g, h \in \text{Aut}(T)$ and $u, v \in X^*$, we have $g|_{uv} = g|_u|_v$, $gh|_u = g|_uh|_{g(u)}$ and $g^{-1}|_u = (g|_{g^{-1}(u)})^{-1}$.  

We define the map  $\psi$ from $\text{Aut}(T)$ to the wreath product $\text{Aut}(T)\wr S_d$ by $$\psi(g) = (g|_1, g|_2, \ldots, g|_d)\lambda_g,$$ where $S_d$ is the permutation group on $X$, and $\lambda_g$ is the induced action of $g$ on the set $X$. For $g \in \text{Aut}(T)$, the expression $\psi(g)$ is called the wreath recursion of $g$. Clearly, $\psi$ is an isomorphism. Accordingly, we identify $\text{Aut}(T)$ with $\text{Aut}(T) \wr S_d$ and we will often write $g$ in place of $\psi(g)$ so that $g = (g|_1, g|_2, \ldots, g|_d)\lambda_g$. 

Note that any subgroup $G$ of $\text{Aut}(T)$, denoted $G \le \text{Aut}(T)$, acts on the tree $T$. A subgroup $G \le \text{Aut}(T)$ is said to be spherically transitive if it is transitive on each level of $T$, i.e., for each $k$ and for all $u, v \in X^k$, there exists an element $g \in G$ such that $g(u) = v$. 

For each level $k$, the $k$-th level stabilizer, denoted by $\text{St}_G(\widehat{k})$, is the stabilizer of the set $X^k$, i.e., $\bigcap_{u \in X^k} \text{St}_G(u)$, the intersection of the stabilizers of vertices in the $k$-th level. 
The normal subgroup $\text{St}_G(\widehat{k})$ is precisely the kernel of the induced action of $G$ on the finite set $X^k$, and hence it has finite index in $G$.

The rigid stabilizer of a vertex $u$, denoted by $\text{Rist}_G(u)$, is the subgroup of $G$ containing all the elements that act trivially on the complement of the subtree $T_u$. The $k$-th level rigid stabilizer, denoted by $\text{Rist}_G(\widehat{k})$, is the subgroup $\left\langle \bigcup_{|u| = k}\text{Rist}_G(u)\right\rangle$.

\begin{definition}
	Let $G$ be a spherically transitive subgroup of $\text{Aut}(T)$. The group $G$ is said to be a branch group if the index of $\text{Rist}_G(\widehat{k})$ in $G$ is finite for every level $k$. The group $G$ is called a weakly branch group if, for all $k$, $\text{Rist}_G(\widehat{k})$ is nontrivial, and hence $|\text{Rist}_G(\widehat{k})| = \infty$.
\end{definition}

\begin{definition}
	A subgroup $G \leq \text{Aut}(T)$ is said to be self-similar if $g|_u \in G$, for every $g \in G$ and $u \in X^*$. 
\end{definition}
Notice that for a self-similar group $G \leq \text{Aut}(T)$ the restriction of the map $\psi$ to $G$ $$\psi: G \to G \wr S_d$$ embeds $G$ into the wreath product $G \wr S_d$.

Let $G$ be a self-similar group and $g \in \text{St}_G(\widehat{1})$. The wreath recursion of $g$ is given by $\psi(g) = (g|_1, g|_2, \ldots, g|_d)$ and the induced homomorphism $$\psi_1: \text{St}_G(\widehat{1}) \to \stackrel{d}{G\times \cdots \times G}$$ is an embedding. Due to self-similarity of $G$, the homomorphism $\psi_1$ extends to all $k \in \mathds{N}$, the set of positive integers, in a natural way such that $$\psi_k: \text{St}_G(\widehat{k}) \to \stackrel{d^k}{G\times \cdots \times G}$$ is injective for all $k \in \mathds{N}$. For convenience, we often identify an element $g \in \text{St}_G(\widehat{k})$ with its image $\psi_k(g)$. For any vertex $u \in X^*$ we can define the projection $\pi_u: \text{St}_G(u) \to G$ by $\pi_u(g) = g|_u$, where $g \in \text{St}_G(u)$. Clearly, for every $u \in X^*$, $\pi_u$ is a homomorphism.

\begin{definition}
	A self-similar group $G \leq \text{Aut}(T)$ is said to be fractal if  $\pi_u(\text{St}_G(u)) = G$, for all vertices $u$ in $T$.
\end{definition}

To prove that a group is fractal, it suffices to show that the above condition is satisfied by the vertices of the first level of the tree, as per the following proposition.

\begin{proposition}[\cite{albizuri2016}]\label{f.level}
	If a self-similar group $G \leq \textup{Aut}(T)$ is transitive on the first level and $\pi_{x}(\textup{St}_{G}(x)) = G$ for some $x \in X$, then $G$ is fractal and spherically transitive.
\end{proposition}

\begin{definition}
	A self-similar group $G \leq \text{Aut}(T)$ is said to be weakly regular branch over a nontrivial subgroup $K \leq G$ if $G$ is spherically transitive and $$\stackrel{d}{K\times \cdots \times K} \le \psi(K \cap \text{St}_G(\widehat{1}))$$
	
	If, additionally, $K$ is of finite index in $G$, then $G$ is said to be regular branch over $K$. 
\end{definition}

\begin{definition}
	A self-similar group $G \leq \text{Aut}(T)$ is said to be contracting if there exists a finite subset $S \subseteq G$ satisfying the following: for every $g \in G$ there is a $k \in \mathds{N}$ such that $g|_u$ belongs to $S$, for all vertices $u$ in $T$ with $|u| \geq k$; otherwise, it is said to be non-contracting.
\end{definition}

\begin{proposition}[\cite{Davis2014}]\label{contracting}
	Let $G \leq \text{Aut}(T)$ be a self-similar group. Suppose that there exist $g\in G$ and $u\in X^{*}$ such that,
	$g|_{u} = g$, $g(u) = u$, and $g$ has infinite order, then $G$ is non-contracting.
\end{proposition}

We use the following notations in the rest of the paper. For elements $g, h$ of a group, we write $g^h := h^{-1}gh$ to denote the conjugate of $g$ by $h$, and  $[g, h] := g^{-1}h^{-1}gh$ to denote the commutator of $g$ and $h$. 

For a set $A$, let $\tilde{A} = A \cup A^{-1}$, where $A^{-1} = \{a^{-1}\ :\ a \in A\}$, the set of formal inverses of elements of $A$. Let $w$ be a word over $\tilde{A}$. For $p \in A$, we write $|w|_p$ to denote the exponent sum of $p$ in $w$. Further, $|w|_{A} = \sum_{a \in A} |w|_a$. For any $j \in \mathds{N}$, we write $\overline{j} \in \{1, \ldots, d\}$ such that $\overline{j} \equiv j\  (\text{mod}\; d)$.

\section{A Class of Non-Contracting Groups}
\label{gen_class}

For $d \geq 3$, consider the group $G_d \leq \text{Aut}(T)$ generated by the elements of $A = \{a_1, a_2, \dots, a_d\}$ which are defined recursively as follows:
\begin{eqnarray*}
	a_{1}& = &(a_{1}, a_{2}, e, \dots, e) (1\ 2)\\
	a_{2}& = &(e, a_{2}, a_{3}, e, \dots , e)(2\ 3)\\
	&\vdots&\\
	a_{d-1}& = &(e, \dots , e, a_{d-1}, a_d)(d-1\ d)\\
	a_{d}& = &(a_{1}, e, \dots, e, a_{d})(d\ 1)
\end{eqnarray*}

Clearly, $G_d$ is a self-similar group. In this section, we prove that $G_d$ is a non-contracting group and, further, study certain properties of their elements. We begin with the following remarks regarding the elements of $G_d$, which are useful in the sequel.

\begin{remark}\label{w_len_twice}
	For a word $w$ over $A$, if the wreath recursion of $w$ is $(w_1, w_2, \ldots, w_d)\lambda_w$, then $|w_1w_2 \cdots w_d| = 2|w|$.
\end{remark}

\begin{remark}\label{exp_com_len_d}
	For a word $w$ over $\tilde{A}$, let the wreath recursion of $w$ be $(w_1, w_2, \ldots, w_d)\lambda_w$. Then we have the following:
	\begin{enumerate}[(i)]
		\item\label{exp_com_twice} $|w_1w_2 \cdots w_d|_A = 2|w|_A$.
		\item For $1 \leq i \leq d$, if $|w|_{a_i} = s_i$, then $|w_1w_2 \cdots w_d|_{a_i} = s_{i} + s_{\overline{i - 1}}$.
		
		\item \label{exp_com_len_odd}  
		If $d$ is odd and  $|w_1w_2 \cdots w_d|_{a_i} = t_i$, for $1 \le i \le d$, then \[|w|_{a_i} = \frac{1}{2}\sum_{j = 1}^{d}t_j - \sum_{j = 1}^{\frac{d - 1}{2}}t_{\overline{i + 2j}}\] for all $i$. 
		
		\item\label{odd_zero} In particular, when $d$ is odd, if $|w_1w_2 \cdots w_d|_{a_i} = 0$ for all $i$, then $|w|_{a_i} = 0$, for all $i$.
	\end{enumerate}
\end{remark}

\begin{remark}\label{even_relator}
	If $d$ even, Remark \ref{exp_com_len_d}(\ref{odd_zero}) does not hold, in general. For instance, consider the word $w  = a_2a_1a_3^{-1}a_2a_1^{-1}a_4a_2^{-1}a_1^{-1}$ in $G_4$. Note that $w|_1 = abb^{-1}a^{-1}$, $w|_2 = bd^{-1}db^{-1}$, $w|_3 = cbb^{-1}c^{-1}$ and $w_4 = c^{-1}ca^{-1}a$. Although $|w|_i|_{a_j} = 0$, for all $1 \le i, j \le 4$, we have $|w|_{a_i} \neq 0$ for any $i$.
\end{remark}

The following lemma is useful to prove $G_d$ is non-contracting. 

\begin{lemma}\label{ord_infi}
	The order of the product $a_{1}a_{2}\cdots a_{d}$ is infinite in $G_d$.
\end{lemma}

\begin{proof}
	Let $g = a_{1} a_{2}\cdots a_{d}$ and $h = a_{2} a_{1} a_{d}a_{d-1}\cdots a_{4}a_{3}$. Note that the section $g|_{2} = a_{1}|_{2} a_{2}|_{1}a_{3}|_{1}\cdots a_{d}|_{1} = a_{2}a_{1}$ and, for each $i \geq 3$, the section \[g|_{i} = a_{1}|_{i} a_{2}|_{i} a_{3}|_{i}\cdots a_{i-1}|_{i} a_{i}|_{i-1} \cdots a_{d}|_{i-1} = a_{i}.\]
	The action of $g$ on the first level of $T$ is given by the permutation $(1\ 2)(2\ 3) \cdots (d\ 1)$, i.e., $(d\ d-1\ldots 3\ 2)$.
	Therefore, $g^{d-1}\in \text{St}_{G_d}(\widehat{1})$ and
	$g^{d-1}|_{2} = g|_{2} g|_{d} g|_{d-1} \cdots g|_{3} = a_{2} a_{1} a_{d} a_{d-1} \cdots a_{3} = h$.
	
	Also, $h|_{1} = a_{2}|_{1}  a_{1}|_{1} a_{d}|_{2} a_{d-1}|_{2} \cdots a_{3}|_{2} = a_{1}$,
	$h|_{2} = a_{2}|_{2} a_{1}|_{3}a_{d}|_{3}a_{d-1}|_{3} \cdots a_{3}|_{3} = a_{2} a_{3}$,
	for $4 \leq i\leq d-1$, $h|_{i} = a_{2}|_{i} a_{1}|_{i} a_{d}|_{i}\cdots a_{i}|_{i} a_{i-1}|_{i+1} \cdots a_{3}|_{i+1} = a_{i}$, and $h|_{d} = a_{2}|_{d} a_{1}|_{d} a_{d}|_{d} a_{d-1}|_{1} \cdots a_{3}|_{1} = a_{d}.$ The action of $h$ on the first level of $T$ is given by the permutation $(2\ 3)(1\ 2)(d\ 1)(d-1\ d)\cdots (3\ 4)$, i.e., $(1\ 2\ 4\ 5 \ldots d)$.
	Therefore, $h^{d-1}\ \in\ \text{St}_{G_d}(\widehat{1})$ and $h^{d-1}|_{1} = h|_{1} h|_{2}h|_{4} \cdots h|_{d} = a_{1} a_{2} a_{3} a_{4}\cdots a_{d} = g.$
	
	Accordingly, for $j \ge 2$, note that
	\begin{align*}
		g^{(d-1)^{j}}|_{2} & =  (g^{d-1}|_{2})^{(d-1)^{j-1}} = h^{(d-1)^{j-1}}, \ \text{and}\\
		h^{(d-1)^{j}}|_{1} & =  (h^{d-1}|_{1})^{(d-1)^{j-1}} = g^{(d-1)^{j-1}}.
	\end{align*}
	If possible, suppose the order of $g$ is $r < \infty$. Then $d-1$ divides $r$; because if $d-1 \notdivides r$ then $g^r \notin \text{St}_{G_d}(\widehat{1})$. Let $r = m(d-1)^n$ for some $m,n \in \mathds{N}$ such that $d-1 \notdivides m$. Then, if $n$ is even,   
	$g^{m(d-1)^n}|_{\underset{\text{length} \ n}{2121\dots 21}} = g^m \notin \text{St}_{G_d}(\widehat{1})$, and, if $n$ is odd, $g^{m(d-1)^n}|_{\underset{\text{length} \ n}{2121\dots 2}} = h^m \notin \text{St}_{G_d}(\widehat{1})$. In both the cases we arrive at a contradiction to the fact that any section of $g^r$ must be in $\text{St}_{G_d}(\widehat{1})$. Hence, the order of the product $a_{1}a_{2}\cdots a_{d}$ is infinite.
\end{proof}

\begin{theorem}\label{th_non-con}
	The group $G_{d}$ is non-contracting.
\end{theorem}

\begin{proof}
	Note that the section $(a_{1}a_{2}\cdots a_{d})|_{1} = a_{1}a_{2}\cdots a_{d}$ and $a_{1}a_{2}\cdots a_d(1) = 1$. Also, by Lemma \ref{ord_infi}, since the order of $a_{1}a_{2}\cdots a_{d}$ is infinite,  the group $G_{d}$ is non-contracting by Proposition \ref{contracting}.
\end{proof}

We now establish the orders of certain elements of $G_d$ and show that any nonempty word over $A$ does not represent the identity element of $G_d$ (see Theorem \ref{freerelator_d}).

\begin{lemma}\label{order_infinite}
	For $1 \leq i \leq d$, the order of the product $a_{1}a_{2}\cdots a_{i - 1}a_{i}a_{i - 1}\cdots a_{3}a_{2}$ is infinite in $G_d$.
\end{lemma}

\begin{proof}
	For $1 \leq i \leq d$, let $h_i = a_{1}a_{2}\cdots a_{i - 1}a_{i}a_{i - 1}\cdots a_{3}a_{2}$. First we show that the order of $h_d$ is infinite. If possible, suppose the order of $h_d$ is finite, say $r$. Note that the action of $h_d$ on the first level of $T$ is given by the permutation $(1\ 2)(2\ 3) \cdots (d\ 1)(d\ d-1)\cdots (3\ 2)$, i.e., $(d\ d-1\ldots 3\ 2)(2\ 3\ \ldots d)$, which is identity permutation. Therefore, $h_d \in \text{St}_{G_d}(\widehat{1})$. Further, note that $h_d|_1 = a_{1}a_{2}\cdots a_{d}$. Accordingly, $e = h_d^r|_1 = (h_d|_1)^r = (a_{1}a_{2}\cdots a_{d})^r$, which contradicts Lemma \ref{ord_infi} that the order of $a_{1}a_{2}\cdots a_{d}$ is infinite. Hence, the order of $h_d$ is infinite.
	
	We now prove the result through the following statement: for $1 \le i \le d-1$, the order of $h_{i+1}$ is infinite implies that the order of $h_i$ is infinite. Suppose the order of $h_{i+1}$ is infinite.
	
	Note that the sections $h_i|_{1} = a_{1}|_{1} a_{2}|_{2} \cdots a_{i}|_{i}a_{i - 1}|_{1}\cdots a_{3}|_{1}a_{2}|_{1} = a_{1} a_{2} \cdots a_{i}$, $h_i|_{2} = a_{1}|_{2} a_2|_{1}\cdots a_{i}|_{1} a_{i - 1}|_{1} \cdots a_{2}|_{1} = a_{2}$, and $h_i|_{i+1} = a_{1}|_{i + 1}\cdots a_{i}|_{i + 1} a_{i - 1}|_{i} \cdots a_{2}|_{3}$ $= a_{i + 1}a_{i} \cdots a_{3}$.
	Also, the action of $h_i$ on the first level of $T$ is given by the permutation $(1\ 2) \cdots (i\ i + 1) (i - 1\ i) \cdots (2\ 3)$, i.e., $(i + 1\ 2\ 1)$.
	Therefore, $h_{i}^{3}\in \text{St}_{G_d}(\widehat{1})$. Observe that
	$h_{i}^{3}|_{1} = h_{i}|_{1} h_{i}|_{i+1} h_i|_{2} = a_{1}a_{2} \cdots a_{i}a_{i+1}a_{i}a_{i-1}\cdots a_{3}a_{2} = h_{i + 1}$.
	
	If possible, suppose the order of $h_{i}$ is $s < \infty$. Then $3$ divides $s$; because if $3 \notdivides s$, then $h_{i}^s \notin \text{St}_{G_d}(\widehat{1})$. Let $s = 3t$ for some $t \in \mathds{N}$. Thus, we have $e = h_{i}^s|_1 = (h_{i}^{3}|_1)^t = h_{i + 1}^t$, which contradicts our assumption that the order of $h_{i + 1}$ is infinite. Hence, the order of $h_{i}$ is infinite.
\end{proof}

\begin{lemma}\label{ord_infinite}
	For $1 \leq i < d$, the order of the product $a_{1}a_{2}\cdots a_{i}$ is infinite in $G_d$.
\end{lemma}

\begin{proof}
	For $1 \leq i < d$, let $g_i = a_{1} a_{2}\cdots a_{i}$. Note that the section $g_i|_{1} = a_{1}|_{1} a_{2}|_{2} \cdots a_{i}|_{i}$ $= a_{1} a_{2} \cdots a_{i}$, and, for $2 \leq j \leq i+1$, the section $g_i|_{j} = a_{1}|_{j} \cdots a_{j-1}|_{j} a_{j}|_{j-1} \cdots a_{i}|_{j-1}$ $= a_{j}$.
	Also, the action of $g_i$ on the first level of $T$ is given by the permutation $(1\ 2)(2\ 3) \cdots (i\ i + 1)$, i.e., $(i + 1\ i\ldots 2\ 1)$.
	Thus, $g_{i}^{i + 1}\in \text{St}_{G_d}(\widehat{1})$. Note that
	$g_{i}^{i + 1}|_{1} = g_{i}|_{1} g_{i}|_{i+1} g_i|_{i} \cdots g_i|_{2} = a_{1}a_{2} \cdots a_{i}a_{i+1}a_{i}a_{i-1}\cdots a_{3}a_{2} = h\ (\textit{say})$.
	
   If possible, suppose the order of $g_{i}$ is finite, say $r$. Then $i + 1$ divides $r$; because if $i + 1 \notdivides r$ then $g_{i}^r \notin \text{St}_{G_d}(\widehat{1})$. Let $r = (i + 1)s$, for some $s \in \mathds{N}$. Thus, we have $e = g_{i}^r|_1 = (g_{i}^{i + 1}|_1)^s = h^s$, which contradicts Lemma \ref{order_infinite} that the order of $h$ is infinite. Hence, the order of $g_{i}$ is infinite.
\end{proof}

\begin{corollary}\label{ord_of_gen}
	The order of $a_1$ is infinite in $G_d$. Hence, by symmetry, the order of $a_i$ is infinite in $G_d$, for $2 \le i \le d$.
\end{corollary}

\begin{theorem}\label{freerelator_d}
	If $w$ is a nonempty word over $A$, then $w$ does not represent the identity element of $G_d$.
\end{theorem}

\begin{proof}
	We prove the statement by induction on the length $|w|$. For $1 \le i \le d$, since $a_i$ is non-identity element of $G_d$, the statement is true for $|w| = 1$. Let $w$ be a word over $A$ of length $n+1$. If $w$ is of the form $a_{i}^{n+1}$, for some $i \in \{1, \ldots, d\}$, then the statement follows from Corollary \ref{ord_of_gen}. Otherwise, $w$ contains a subword of the form $a_ia_j$, for $i \ne j$. In the following, we observe that a section of $a_ia_j$ has length one by considering $j$ in the cases $\overline{i + 1}, \overline{i - 1}$, or otherwise. Note that, for $1 \le i \le d$, 
	\[\begin{array}{rcl}
		a_ia_{\overline{i + 1}}|_l &=& \left\{ \begin{array}{ll}
												a_ia_{\overline{i + 1}}, & \text{if }\ l = i;\\
												a_{l}, & \text{if}\ l \in \{\overline{i + 1}, \overline{i + 2}\};\\
												e, &  \text{otherwise}.
												\end{array}
										\right.  \\
									&&\\
		a_{i}a_{\overline{i - 1}}|_l &=& \left\{ \begin{array}{ll}
												a_l, & \text{if}\ l \in \{i, \overline{i - 1}\};\\
												a_{\overline{i + 1}}a_{i}, & \text{if }\ l = \overline{i + 1};\\
												e, &  \text{otherwise}.
											\end{array}
										\right.
	\end{array}\]
	For $d > 3$, if $1 \le i, j \le d$ with $j \notin \{\overline{i - 1}, i, \overline{i + 1}\}$, we have
	\[\begin{array}{rcl}
		a_ia_{j}|_l &=& \left\{ \begin{array}{ll}
									a_l, & \text{if }\ l \in \{i, j, \overline{i + 1}, \overline{j + 1}\};\\
									e, &  \text{otherwise}.
								\end{array}
						\right.
	\end{array}\]
  Thus, there exists at least one $l$ such that $|a_ia_j|_l| = 1$. Consequently, there exists $k$ such that  $1 \leq |w|_{k}| < |w|$. Therefore, by inductive hypothesis $w|_k$ is non-identity in $G_d$. Hence, $w$ is non-identity in $G_d$.
\end{proof}

\section{A Class of Fractal and Weakly Regular Branch Groups}
\label{odd_class}

We now restrict the values of $d$ to odd integers and consider the corresponding class of non-contracting groups defined in Section \ref{gen_class}. We establish some important properties of the groups in this class, including just infiniteness, and word problem. Further, we show that a group in the class is fractal, spherically transitive, and weakly regular branch group over its commutator subgroup.  

\begin{theorem}\label{exponent_sum_d}		
	Let $d$ be odd. Suppose $w$ is a word over $\tilde{A}$ representing the identity element of $G_d$. Then $|w|_p = 0$, for all $p \in A$.
\end{theorem}

\begin{proof}
	We prove the statement by induction on the length $|w|$. In view of Theorem \ref{freerelator_d}, if $w$ is a word over $A$ or over $A^{-1}$, then $w$ is a non-identity element of $G_d$. It can be easily verified that every reduced word $w$ such that $1\leq |w|\leq 2$ represents a non-identity element of $G_d$. Suppose the statement is true for all words of length $\leq n$. Let us consider a word $w$ of length $n + 1$ that represents the identity element in $G_d$. Then clearly $w$ contains at least one symbol from $A$ and one symbol from $A^{-1}$. Note that we can get a circular permutation on the symbols of $w$ obtained by applying conjugation on $w$. Let $w'$ be such a conjugate of $w$ such that $|w| = |w'|$ and $w' = w_0 a_i^{-1} a_j$, for some $i, j$. 
	
	If $i = j$, then $w^{\prime} = w_0$. Since $|w_0| < |w|$, by inductive hypothesis, $|w_0|_p = 0$ for all $p \in A$. Thus, $|w'|_p = 0$ for all $p \in A$. Since $w'$ is a conjugate of $w$, we have $|w|_p = 0$ for all $p \in A$. 
	
	Suppose $i \ne j$. In the following, we observe that for each $k\in \{1, \ldots, d\}$  section $a_i^{-1}a_j|_k$ has length at most one by considering $j$ in the cases $\overline{i + 1}, \overline{i - 1}$, or otherwise. Note that, for $1 \le i \le d$,
	\[\begin{array}{rcl}
		a_i^{-1}a_{\overline{i + 1}}|_l &=& \left\{ \begin{array}{ll}
			a_{i}^{-1}, & \text{if }\ l = \overline{i + 1};\\
			a_{\overline{i + 2}}, & \text{if}\ l = \overline{i + 2};\\
			e, &  \text{otherwise}.
		\end{array}
		\right.  \\
		&&\\
		a_{i}^{-1}a_{\overline{i - 1}}|_l &=& \left\{ \begin{array}{ll}
			a_l, & \text{if}\ l = \overline{i - 1};\\
			a_{\overline{i + 1}}^{-1}, & \text{if }\ l = i;\\
			e, &  \text{otherwise}.
		\end{array}
		\right.
	\end{array}\]
	For $d > 3$, if $1 \le i, j \le d$ with $j \notin \{\overline{i - 1}, i, \overline{i + 1}\}$, we have
	\[\begin{array}{rcl}
		a_i^{-1}a_{j}|_l &=& \left\{ \begin{array}{ll}
			a_{\overline{i + 1}}^{-1}, & \text{if }\ l = i;\\
			a_{i}^{-1}, & \text{if }\ l = \overline{i + 1};\\
			a_{l}, & \text{if }\ l \in \{j, \overline{j + 1}\};\\
			e, &  \text{otherwise}.
		\end{array}
		\right.
	\end{array}\]
	Thus, for all $l \in \{1, \ldots, d\}$, 
	\begin{equation}\label{len_redu}
		|a_i^{-1}a_j|_l| \leq 1 
	\end{equation}
	For $l \in \{1, \ldots, d\}$, note that
	$$|w'|_l| = |w_0a_i^{-1}a_j|_l| \leq |w_0|_l|+|a_i^{-1}a_j|_{w_0(l)}| \leq |w_0| + 1 < |w_0| + 2 = |w'|.$$	
	Since $w'$ represents the identity element in $G_d$,  $w'|_l = e$ for all $l$.
	By inductive hypothesis, the exponent sum $|w'|_{l}|_{a_k}$ is zero, for all $k, l$.  Hence, for all $p \in A$, $|w'|_1 w'|_2 \cdots w'|_d|_p = 0$. Since $d$ is odd, by Remark \ref{exp_com_len_d}(\ref{exp_com_len_odd}), we have $|w'|_p = 0$ for all $p \in A$, and hence, $|w|_p = 0$ for all $p \in A$.	
\end{proof}	

\begin{remark}
Theorem \ref{exponent_sum_d} is not true when $d$ is even. In fact, the word $w$ given in Remark \ref{even_relator} is a relator in $G_4$, whereas $|w|_p \ne 0$ for any $p \in A$. 
\end{remark}

In the following theorem we obtain the abelianization of $G_d$ and accordingly show that $G_d$ is not just infinite.

\begin{theorem}\label{quotient}
	For odd values of $d$, the quotient group $G_d/G_d'$ is isomorphic to $\mathds{Z} \times \stackrel{d}{\cdots} \times \mathds{Z}$.
\end{theorem}	

\begin{proof}
	We prove that for any $(i_1, i_2, \ldots, i_d) \neq (0, 0, \ldots, 0)$ with $i_j \in \mathds{Z}$, we have $a_1^{i_1}a_2^{i_2}\cdots a_d^{i_d} \notin G_d'$. On the contrary, suppose $w = a_1^{i_1}a_2^{i_2}\cdots a_d^{i_d} \in G_d'$ for some $i_j \in \mathds{Z}$ ($1 \le j \le d$) with at least one $i_j \ne 0$.  Since every element of $G_d'$ is generated by commutators and their conjugates, $w$ equals a word $w'$ in $G_d$ such that the exponent sum $|w'|_p = 0$ for all $p \in A$.  Thus, $w'w^{-1}$ is a relator in $G_d$ with $|w'w^{-1}|_p \ne 0$ for some $p \in A$, which contradicts Theorem \ref{exponent_sum_d}. Hence, $\{a_1^{i_1}a_2^{i_2}\cdots a_d^{i_d}: i_1, i_2, \ldots i_d \in \mathds{Z}\}$ is set of representatives for the cosets of $G_d/G_d'$ so that $G_d/G_d'$ is isomorphic to $\mathds{Z} \times \stackrel{d}{\cdots} \times \mathds{Z}$.
\end{proof}

\begin{corollary}
	For odd values of $d$, the group $G_d$ is not just infinite.
\end{corollary}

Recall that a self-similar group $G$ acting on $X^*$ is said to be finite-state if for any $g \in G$ the set $\{g|_v : v\in X^*\}$ is finite. The word problem is solvable for finitely generated finite-state self-similar groups, but the algorithm runs in exponential time (for instance, see \cite{Steinberg2015}).  For contracting groups there is a polynomial-time algorithm \cite{savchuk2003word}. Although $G_d$ is non-contracting, we now present an efficient algorithm for the group $G_d$, when $d$ is odd, which is inspired by the algorithm given in \cite{grigorchuk2002torsion} for a contracting group. 

\begin{theorem}
	For odd values of $d$, the word problem for $G_d$ is solvable via an efficient algorithm.
\end{theorem}

\begin{proof} Consider the following procedure Id-Br($u$) for which a word $u$ over $\tilde{A}$ is an input.
\begin{center}
\begin{minipage}{12cm}
	\noindent Id-Br($u$) =
	\begin{enumerate}[(i)]
		\item If $u \in \text{St}_{G_d}(\widehat{1})$, then go to step (ii). Otherwise, $u \ne e$. 
		\item If $|u|_p = 0$, for all $p \in A$, then go to step (iii). Otherwise, $u \ne e$.
		\item If $u$ has a subword of the form $p^{-1}q$, for some $p, q \in A$, then go to step (v). Otherwise, go to step (iv).
		\item Consider a conjugate $u'$ of $u$ through a cyclic shift such that $u'$ has a subword of the form $p^{-1}q$, for some $p, q \in A$. Replace $u$ by $u'$ and go to step (v). 
		\item Compute the wreath recursion of $u$, say $(u_1, u_2, \ldots, u_d)$. 
	\end{enumerate}
\end{minipage}
\end{center}

Let $w (\ne \varepsilon)$ be a word over $\tilde{A}$. 
If the wreath recursion of $w$ is $(w_1, w_2, \ldots, w_d)\lambda_w$, then note that $w = e$ if and only if $w_i = e$, for all $i \in \{1, 2, \ldots, d\}$ and $\lambda_w$ is the identity permutation. Give input $w$ to the procedure Id-Br($\;\;$) and iteratively call the procedure on each branch (obtained in step 5) if it is not empty. In any iteration, if the procedure returns that the input is not $e$, then the given word $w \ne e$. Otherwise, after a finite number of iterations the process results that $w = e$. The process converges because we have the length reduction of each branch. Indeed, for $p, q \in A$, we have $|p^{-1}q|_i| < |p^{-1}q|$, for all $i \in \{1, 2, \ldots, d\}$ (cf. Equation \ref{len_redu}). 
\end{proof}

\begin{theorem}\label{g_d_sp_tran}
	For odd values of $d$, the group $G_{d}$ is fractal and spherically transitive.
\end{theorem}

\begin{proof}
	Consider the element  $g = a_1 a_2 \cdots a_{d-1} \in G_{d}$ and note that the action of $g$ on the first level of $T$ is given by the permutation $\lambda_{g} = (1\ 2)(2\ 3) \cdots (d-1\ d) = (d \ d-1 \dots 1)$. Therefore, for all $x, y \in X$, $\lambda_{g}^i(x) = y$ for some $i$. Hence, $G_{d}$  is transitive on the first level of $T$. In view of  Proposition \ref{f.level}, we show that $\pi_{1}(\textup{St}_{G_{d}}(1)) = G_{d}$ so that $G_{d}$ is fractal and spherically transitive. For this, in the following, we explicitly construct elements within the subset $\text{St}_{G_{d}}(\widehat{1})$ whose images under $\pi_{1}$ give the generators $a_i$'s of $G_{d}$. 
	
	For $2 \leq i \leq d-1$, let $s_i = (a_{i}^{2})^{g_{i-1}}$, where $g_{i-1} = a_{i-1} \cdots a_1$. For all $i$, clearly $s_i \in \text{St}_{G_{d}}(\widehat{1})$. Note that the section 
	$g_{i-1}|_i = a_{i-1}|_i a_{i-2}|_{i-1} a_{i-3}|_{i-2} \cdots a_1|_{2} = a_{i} \cdots a_2$
	and, since $g_{i-1}^{-1}(1) =  i$, we have $g_{i-1}^{-1}|_1 = (g_{i-1}|_i)^{-1} = {a_2}^{-1} \cdots {a_{i}}^{-1}$.
	Further, note that $\pi_1(s_2) = s_2|_1 = a_3a_2$ and, for $3 \leq i \leq d-1$,
	$\pi_1(s_i) = s_i|_1 = g_{i-1}^{-1}|_1 a_{i}^{2}|_i g_{i-1}|_i =
	({a_2}^{-1} \cdots {a_{i}}^{-1}) (a_{i}a_{i+1})(a_i \cdots a_2) = a^{-1}_{2} \cdots a^{-1}_{i-1}a_{i+1}\ \cdots a_{2}.$

	Let $h = a_2 a_{3} \cdots a_{d} a_1$. Note that $\lambda_h = (2\ 3)(3\ 4) \cdots (d\ 1)(1\ 2) = (d\ d-1\ \dots \break \dots\ 3 \ 1)$ and hence $h^{d-1} \in \text{St}_{G_{d}}(\widehat{1})$.  To calculate $\pi_{1}(h^{d-1})$, first observe that
		$h|_1 = a_2|_1 a_3|_1 \cdots a_{d}|_1 a_1|_{d} = a_1$, 
		$h|_3 = a_2|_3 a_3|_2 \cdots a_{d}|_2 a_1|_2 = a_3 a_2$, 
	and, for $4 \leq j \leq d$,
	$h|_j = a_2|_j \cdots a_{j-1}|_j a_{j}|_{j-1} \cdots a_{d}|_{j-1} a_1|_{j-1} = a_j$.
	Thus, 
	$\pi_{1}(h^{d-1}) = h|_{1} h|_{d} h|_{d-1} \cdots h|_{3} \break = a_{1} a_{d} a_{d-1} \cdots a_{3} a_{2}$.

	Now note that for $1\leq i \leq \frac{d-1}{2}$, $\pi_{1}(s_{2}s_{4} \cdots s_{2i}) = a_{2i+1}a_{2i} \cdots a_{2}$.
	Accordingly,
	$$\pi_{1}(h^{d-1}(s_{2}s_{4} \cdots s_{d-1})^{-1}) = a_1 \in \pi_{1}(\text{St}_{G_{d}}(\widehat{1})).$$ 
	Let $s_1 = (s_{2}s_{4} \cdots s_{d-1})h^{-(d-1)}a_{1}^{2}$. Then $\pi_{1}(s_1) = a_2$ as $\pi_{1}( a_{1}^{2}) = a_1a_2$.
	Further, for $1\leq i \leq \frac{d-1}{2}$,  $$\pi_{1}(s_{1}s_{3} \cdots s_{2i-1}) = a_{2i}a_{2i-1} \cdots a_{2}.$$ 	
	A direct computation gives
	\begin{eqnarray*}
		a_{2i} &=& \pi_{1}\left((s_{1} s_{3} \cdots s_{2i-1})(s_{2} \cdots s_{2i-2})^{-1}\right),\ 2 \leq i \leq \frac{d-1}{2}; \\
		a_{2i+1} &=& \pi_{1}\left((s_{2} s_{4} \cdots s_{2i})(s_{1} s_{3} \cdots s_{2i-1})^{-1}\right),\  1 \leq i \leq \frac{d-1}{2}.
	\end{eqnarray*} 
	Hence, for all $i \in \{1, \dots, d\}$, $a_i \in \pi_{1}(\text{St}_{G_{d}}(\widehat{1}))$. 
\end{proof}

We use the following notation in the proof of Theorem \ref{g_d_we_re_br} for presenting it in an elegant manner.  
\begin{notation}
	Let the wreath recursion of an element $g = (g_1, \ldots,g_d)\lambda_g$. If $g_i = e$ for most of the components, then $g$ is simply denoted by $\ldbr \underset{i_1}{g_{i_1}}, \ldots, \underset{i_n}{g_{i_n}}\rdbr \lambda_g$, where $g_{i_l}$ is non-identity element appearing at $i_l$-th component. For example, if $g = (a, e, e, b, e, e, c)\lambda_g$, then $g$ can be denoted by $\ldbr \underset{7}{c}, \underset{1}{a}, \underset{4}{b}\rdbr \lambda_g$, irrespective of the order in which the non-identity elements appear in wreath recursion of $g$, as their positions are given. 
\end{notation}

\begin{theorem}\label{g_d_we_re_br}
	For odd values of $d$, the group $G_{d}$ is weakly regular branch over its commutator subgroup $G_{d}^{'}$.
\end{theorem}

\begin{proof}
	In view of Theorem \ref{g_d_sp_tran}, it is sufficient to prove $G_{d}^{'} \times \dots \times G_{d}^{'} \leq \psi_1 (\text{St}_{G_{d}}(\widehat{1})\ \cap\ G_{d}^{'})$. For this, we construct elements in $\text{St}_{G_{d}}(\widehat{1})\ \cap\ G_{d}^{'}$ whose images under $\psi_1$ give all the generators of $G_{d}^{'} \times \dots \times G_{d}^{'}$. 
	
	For $i , j \in \{1,2, \dots, d\}$, note that the commutators $[a_{i}, a_{j}] = e$, when\break $|i-j|\neq 1$ and $|i-j| \neq d-1$. Hence, the commutator group $G_{d}^{'}$ is generated by $[a_{1}, a_{2}], [a_{2}, a_{3}], \ldots, [a_{d-1}, a_{d}], [a_{d}, a_{1}]$ and their conjugates, written $[a_i, a_j]^g$, by any element $g  \in G_{d}$. 
	
	For $1 \le i \le d$, let $\beta_{i} = [a_{i}, a_{\overline{i+1}}]$. We compute the wreath recursion of the following elements in the succession:
	\begin{align*}
		\beta_{i} &= \ldbr \underset{i}{a_{\overline{i+1}}^{-1}}, \underset{\overline{i+1}}{a_{\overline{i+1}}}\rdbr(i\ \overline{i+1}\ \overline{i+2})\\
		\beta_{i} \beta_{\overline{i+1}} &= \ldbr\underset{i}{a_{\overline{i+1}}^{-1}a_{\overline{i+2}}^{-1}}, \underset{\overline{i+1}}{a_{\overline{i+1}}a_{\overline{i+2}}}\rdbr(i\ \overline{i+2})(\overline{i+1}\ \overline{i+3})\\
		[a_{i}^{2}, a_{\overline{i+1}}] &= \ldbr\underset{\overline{i+1}}{a_{i}^{-1} a_{\overline{i+1}}^{-1}}, \underset{\overline{i+2}}{a_{i}a_{\overline{i+1}}}\rdbr\\
		[a_{i}^{2}, a_{\overline{i+1}}]^{a_{i}} &= \ldbr\underset{i}{a_{\overline{i+1}}^{-1}a_{i}^{-1}}, \underset{\overline{i+2}}{a_{i}a_{\overline{i+1}}}\rdbr 
	\end{align*}  
    
	For $1 \leq i \leq d$, consider $\xi_i = [a_{\overline{i+1}}^{2}, a_{\overline{i+2}}](\beta_{i}\beta_{\overline{i+1}})^{-1} \in G_{d}$. Note that the wreath recursion of $\xi_i$ is given by 
	\begin{equation}\label{h_i}
		\xi_{i} = (e, \dots, e)(i\ \overline{i+2})(\overline{i+1}\ \overline{i+3}).
	\end{equation}

	In case of the group $G_3$, let  $g_i = \xi_i$ (for $1\leq i\leq 3$). In case of $G_{d}$ (for $d \geq 5$), let $g_i = (\xi_{\overline{i+1}}\xi_{i})(\xi_{\overline{i+3}}\xi_{\overline{i+2}})\cdots (\xi_{\overline{i+(d-4)}}\xi_{\overline{i+(d-5)}})\xi_{\overline{i+(d-3)}}$ (for $1\leq i\leq d$). For $1\leq j\leq d$, since $\xi_{\overline{j+1}}\xi_{j}$ sends $j$ to $\overline{j+2}$ and fixes $\overline{j+1}$, we have $g_i(i) = \overline{i-1}$ and $g_i(\overline{i+1}) = \overline{i+1}$, for any $1 \le i \le  d$.
	
	Note that, for $1 \leq i \leq d$,
	\begin{align*}
		([a_{i}^{2}, a_{\overline{i+1}}]^{a_{i}})^{\xi_{i}^{-1}} &= \ldbr\underset{i}{a_{i}a_{\overline{i+1}}}, \underset{\overline{i+2}}{a_{\overline{i+1}}^{-1}a_{i}^{-1}}\rdbr\\
		[a_{i}^{2}, a_{\overline{i+1}}]^{g_{\overline{i+1}}} &= \ldbr\underset{i}{a_{i}^{-1} a_{\overline{i+1}}^{-1}}, \underset{\overline{i+2}}{a_{i}a_{\overline{i+1}}}\rdbr
	\end{align*}
	Hence,
	$[a_{i}^{2}, a_{\overline{i+1}}]^{g_{\overline{i+1}}}([a_{i}^{2}, a_{\overline{i+1}}]^{a_{i}})^{\xi_i^{-1}} = \ldbr\underset{i}{[a_{i}, a_{\overline{i+1}}]}\rdbr,\ \textup{for} \ 1 \leq i \leq d.$ 	
	Further, in the wreath recursion of $[a_{i}^{2}, a_{\overline{i+1}}]^{g_{\overline{i+1}}}([a_{i}^{2}, a_{\overline{i+1}}]^{a_{i}})^{\xi_i^{-1}}$, we can move the generators $[a_{i}, a_{i+1}]$ to any $j$-th position  by taking appropriate conjugates by a product of $\xi_{i}$'s. Clearly, these elements are in $\text{St}_{G_{d}}(\widehat{1})\ \cap\ G_{d}'$. 
	
	Thus, the images of the above-produced elements under $\psi_1$ are the generators of $G_{d}' \times \cdots \times G_{d}'$ of the form $(e, \ldots, e, \underset{j}{[a_i, a_{\overline{i+1}}]}, e, \ldots, e)$ having commutators at any $j$-th component. Since $G_{d}$ is fractal, we also get the generators of the form $(e, \ldots, e, \underset{j}{[a_i, a_{\overline{i+1}}]^g}, e, \ldots, e)$ having conjugates of the commutators at any $j$-th component. 
	Hence, 
	for any $(\alpha_{1}, \alpha_{2}, \dots, \alpha_{d}) \in G_{d}' \times \dots \times G_{d}'$, there exists \break $\alpha \in \text{St}_{G_{d}}(\widehat{1})\cap G_{d}'$ such that $\psi_1(\alpha) = (\alpha_{1}, \alpha_{2}, \dots, \alpha_{d})$. 
\end{proof}

In view of Theorem \ref{quotient}, we have the following corollary of Theorem \ref{g_d_we_re_br}.

\begin{corollary}
	 For odd values of $d$, the group $G_d$ is not a regular branch group over $G_d'$.
\end{corollary}

Note that by Equation \ref{h_i}, in the group $G_3$, the order of $\xi_i$ (for $1\leq i\leq 3$) is $3$ and in the group $G_{d}$ ($d\ge 4$ is odd), the order of $\xi_i$ (for $1\leq i\leq d$) is 2. Further, since $G_d$ has infinite order elements (see Section \ref{gen_class}), we have the following remark. 

\begin{remark}
	For odd values of $d$, the group $G_d$ is neither torsion nor torsion-free.
\end{remark}

\section{A Branch Group}
\label{grp_g3}

In this section, we consider $d = 3$ and show that the group $G_d$ is a branch group. This provides an explicit example in the class of non-contracting weakly regular branch groups established in Section \ref{odd_class}.    

Let $A = \{a, b, c\}$ and consider the group $G_3$ generated by $A$. Recall that the wreath recursions of $a, b$ and $c$ are as per the following:
\begin{eqnarray*}
	a &=& (a, b, e)(1\ 2)\\
	b &=& (e, b, c)(2\ 3)\\
	c &=& (a, e, c)(3\ 1)
\end{eqnarray*}

We rephrase Remark \ref{exp_com_len_d}(\ref{exp_com_len_odd}) for $G_3$ and present in the following remark. 
\begin{remark}\label{exp_com_len}
	For a word $w$ over $\tilde{A}$, let the wreath recursion of $w$ be $(w_1, w_2, w_3)\lambda_w$. If $|w_1w_2w_3|_a = t_1$, $|w_1w_2w_3|_b = t_2$ and $|w_1w_2w_3|_c = t_3$, then $|w|_a = \frac{t_1 + t_2 - t_3}{2},\ |w|_b = \frac{t_2 + t_3 - t_1}{2}$ and $|w|_c = \frac{t_3 + t_1 - t_2}{2}$. 
\end{remark}

\begin{remark}
	In the group $G_3$, the elements $\xi_i$, for $1 \le i \le 3$, given in Equation \ref{h_i} coincide and are equal to $(e, e, e)(1\ 3\ 2)$. We denote the element by $\xi$. 
\end{remark}

\begin{remark}\label{len_even_st}
	For $w \in G_3$ with $|w| = 3$, note that $\lambda_w = \lambda_p$ for some $p \in A$. Accordingly, if $|w|$ is odd, then $\lambda_w$ is not identity so that $w \not\in \text{St}_{G_3}(\widehat{1})$. Thus, if $w \in \text{St}_{G_3}(\widehat{1})$, then $|w|$ is even and hence $|w|_A$ is even. 
\end{remark}

\begin{lemma}\label{sec_non_iden}
	Let $p_1, p_2, q_1, q_2 \in A$ with $p_2 \ne q_2$. For each permutation $\sigma$ on $\{1, 2, 3\}$, there exists $j \in \{1, 2, 3\}$ such that the non-identity (i.e., $\ne e$) sections $p_1p_2|_j$ and $q_1q_2|_{\sigma(j)}$ are not identical and they are not of the form $pq$ and $q$, for some $p, q \in A$.   
\end{lemma}

\begin{proof}
	By considering the possibilities of two length words, we prove the result using brute force method. Consider the wreath recursions of all two length words over $A$ as per the following, where $\lambda_1 = (1\ 3\ 2), \lambda_2 = (1\ 2\ 3)$. 
	\[\begin{array}{ccccc}
			ab = (ab, b, c)\lambda_1  &\quad& ba = (a, b, cb)\lambda_2 &\quad& a^2 = (ab, ba, e)\\
			bc = (a, bc, c)\lambda_1  && cb = (ac, b, c)\lambda_2 && b^2 = (e, bc, cb)\\
			ca = (a, b, ca)\lambda_1  && ac = (a, ba, c)\lambda_2 && c^2 = (ac, e, ca)
	\end{array}\]
	We discuss the result in the following three cases. 
	
	Case-I: $p_1 \ne p_2$ and $q_1 \ne q_2$. It is evident from the list that  two among the sections $p_1p_2|_i$ ($1 \le i \le 3$) end with one letter and the third section ends with a different letter. Accordingly, there exists $j$ such that the section $p_1p_2|_j$ is not identical with the section $q_1q_2|_{\sigma(j)}$. 
	
	Case-II: $p_1 = p_2$ and $q_1 = q_2$. Let us consider $a^2$ and $b^2$. Note that $a^2|_1 = ab$ is not identical with the sections $bc$ and $cb$ of $b^2$. In case $\sigma(1) = 1$, then clearly $a^2|_2 = ba$ is not identical with the sections $bc$ and $cb$ of $b^2$. Similarly, for all $p, q \in A$, it can be observed that there exists $j$ such that the section $p^2|_j$ is not identical with the section $q^2|_{\sigma(j)}$.
	
	Case-III: $p_1 = p_2 = p$ and $q_1 \ne q_2$. Note that two among the sections $p^2|_i$ ($1 \le i \le 3$) are of the form $pq$ and $qp$, for $q \in A$. If, for some $i$,  $q_1q_2|_i$ ends with $p$ (or $q$), then other two sections of $q_1q_2$ never end with $q$ (or $p$, respectively). Hence, there exists $j$ such that the section $p^2|_j$ is not identical with the section $q_1q_2|_{\sigma(j)}$.
	
	In all three cases, it is evident that the resultant sections are not of the form $p'q'$ and $q'$, for some $p', q' \in A$.      	 
\end{proof}

\begin{theorem}
The semigroup $S$ generated by $A$ is free of rank $3$.
\end{theorem}

\begin{proof}
	By Theorem \ref{freerelator_d}, if $w$ is a word over $A$, then $w$ cannot be a relator in $G_3$ and hence in $S$. Also, by Theorem \ref{exponent_sum_d}, if $u, v$ are words over $A$ such that $|u|_p \ne |v|_p$ for some $p \in A$, then $uv^{-1} \ne e$ in $G_3$ and hence $u = v$ is not a relation in $S$. 
	
	Accordingly, it is sufficient to show that the group $G_3$ does not have any relation of the form $u = v$ for nonempty and non-identical words $u$ and $v$ over $A$ such that $|u|_p = |v|_p$, for all $p \in A$. On the contrary, suppose such relations exist in $G_3$ and $u' = v'$ is such a relation with $|u'| + |v'|$ is minimum. 
	
	Observe that $|u'| = |v'| \geq 3$ and $u'|_i = v'|_i$, for all $i \in\{1, 2, 3\}$. Assume that, $u'$ is of the form $p_0u_0p_1p_2$ and $v'$ is of the form $q_0v_0q_1q_2$, where $p_0, p_1, p_2, q_1, q_2, q_3 \in A$ and $u_0, v_0$ are words over $A$. Note that $p_0 \ne q_0$; otherwise, $u_0p_1p_2 = v_0q_1q_2$ with $|u_0p_1p_2| + |v_0q_1q_2| < |u'| + |v'|$, which contradicts our assumption. Similarly, $p_2 \ne q_2$. Recall the wreath recursions of $a, b$ and $c$ and note that the sections $p_0|_{i_1} = q_0|_{i_1} = p$,  $p_0|_{i_2} = e$, $q_0|_{i_2} = q$, $p_0|_{i_3} = r$ and $q_0|_{i_3} = e$, for distinct $i_1, i_2, i_3 \in \{1, 2, 3\}$ and $p, q, r \in A$. Consequently, $u'|_{i_1} = pu_1$ and $v'|_{i_1} = pv_1$, for some words $u_1, v_1$ over $A$, $|u'|_{i_2}| + |v'|_{i_2}| < |u'| + |v'|$, and $|u'|_{i_3}| + |v'|_{i_3}| < |u'| + |v'|$.
    
    Note that $u'|_i = (p_0u_0)|_i(p_1p_2)|_{\sigma_1(i)}$, and $v'|_i = (q_0v_0)|_i(q_1q_2)|_{\sigma_2(i)}$, where $\sigma_1 = \lambda_{p_0u_0}$ and $\sigma_2 = \lambda_{q_0v_0}$. Since $p_2 \ne q_2$, by Lemma \ref{sec_non_iden}, there exists $j \in \{1, 2, 3\}$ such that the sections $p_1p_2|_{\sigma_1(j)}$ and $q_1q_2|_{\sigma_2(j)}$ are not identical and they are not of the form $p'q'$ and $q'$, for some $p', q' \in A$. Hence, $u'|_j$ and $v'|_j$ are not identical.
    
    In case $j = i_1$, since $u'|_j = v'|_j$, we have $u_1 = v_1$ and they are not identical with $|u_1| + |v_1| < |u'|_j| + |v'|_j| \le |u'| + |v'|$. If for some $p' \in A$, $|u_{1}|_{p'} \neq |v_{1}|_{p'}$, then we get a contradiction by Theorem \ref{exponent_sum_d}; else, it contradicts our assumption that $|u'|+|v'|$ is minimum.
    
    In case $j = i_2$ or $i_3$, we clearly have $u'|_j = v'|_j$ and they are not identical with $|u'|_{j}| + |v'|_{j}| < |u'| + |v'|$. If for some $p' \in A$, $|u|_{j}|_{p'} \neq |v|_{j}|_{p'}$, then we get a contradiction by Theorem \ref{exponent_sum_d}; else, it contradicts our assumption that $|u'|+|v'|$ is minimum. 
    
Hence, the semigroup $S$ generated by $A$ is free.
\end{proof}

Since the subsemigroup $S$ generated by $A$ in $G_3$ is free, we have the following corollary. 

\begin{corollary}
	The group $G_3$ is of exponential growth.
\end{corollary}

In $G_3$, for $k \ge 1$, consider the set $H_k = \{w \in G_3 \; :\; |w|_{A} \equiv 0\ (\text{mod}\ 2^{k+1})\}$. Clearly, $H_k$ is a normal subgroup of $G_3$. 

\begin{lemma}\label{ai}
	For every $g\in G_3$, there exists $h\in H_k$ and $i \in \mathds{N}$ with $0 \le i  < 2^{k+1}$ such that $g = a^ih$.
\end{lemma}
\begin{proof}
	For $w \in G_3'$, since $|w|_p = 0$ for all $p \in A$, we have $G_3' \le H_k$. Accordingly, for any $g \in G_3$, we have
	\begin{eqnarray*}
		g &=& a^{n_{1}}b^{n_{2}}c^{n_{3}}h_1,\ \text{for some}\ h_1 \in G_3'\\
		&=& a^{n_{1}}b^{n_{2}}c^{n_{3}}(c^{-n_{3}}b^{-n_{2}}a^{n_2 + n_3})h_2,\ \text{where}\ h_2 = (c^{-n_{3}}b^{-n_{2}}a^{n_2 + n_3})^{-1}h_1 \in H_k\\
		&=& a^{(n_{1} + n_{2} + n_{3})}h_2\\
		&=& a^ih,\ \text{for some}\ h \in H_k\ \text{and}\ 0 \leq i < 2^{k+1}\ \text{with}\ i \equiv n_{1} + n_{2} + n_{3} \ (\text{mod}\ 2^{k+1}).
	\end{eqnarray*}
\end{proof}

\begin{theorem}
	For all $k \in \mathds{N}$, $\psi(\textup{Rist}_{G_3}(\widehat{k})) = \stackrel{3^k}{H_k \times \dots \times H_k}$.
\end{theorem}

\begin{proof}
	We prove the statement by induction on $k$. 
	For $k = 1$, it is sufficient to show that $\psi(\text{Rist}_{G_3}(1)) = H_1 \times \{e\} \times \{e\} $ so that $\psi(\text{Rist}_{G_3}(\widehat{1})) = H_1 \times H_1 \times H_1$.
	
	Let $\psi(\text{Rist}_{G_3}(1)) = K \times \{e\} \times \{e\}$. Note that $K$ is a subgroup of $G_3$. We will show that $K = H_1$. Since $G_3$ is weakly regular branch group over $G_{3}'$, we have $G_{3}'\leq K$. Also, note that $c^{-1}a, (ab)^{2}\in K$ because $c^{-1}b^{-1}ac\xi = (c^{-1}a, e, e)$ and $(a^{b}\xi^{2})^{2} = ((ab)^{2}, e, e)$. Since $G_3' \le K$, clearly $a^2b^2 \in K$. By symmetry, we can observe that 
	$a^{-1}b, b^{-1}c, b^{2}c^{2}, a^{2}c^{2} \in K$ and hence $a^{4}, b^{4}, c^{4} \in K$.
	
	Let $w\in H_1$ and $|w|_{a} = n_{1}, |w|_{b} = n_{2}, |w|_{c} = n_{3}$. Thus, $n_{1} + n_{2} + n_{3} = 0\; (\text{mod}\ 4)$. Observe that  
	\begin{eqnarray*}
		 w &=& a^{n_{1}}b^{n_{2}}c^{n_{3}}h_1,\ \text{for some}\ h_1 \in G_3'\\
		 &=& a^{n_{1}}b^{n_{2}}c^{-n_{1}-n_2+4t}h_1,\ \text{for some}\;  t \in \mathds{Z}\\
		 &=& c^{-n_{1}}a^{n_{1}}c^{-n_2}b^{n_{2}}h_2,\ \text{for some}\ h_2 \in K, \text{as}\; G_3' \le K\; \text{and}\; c^4 \in K\\
		 &=& (c^{-1}a)^{n_{1}}(c^{-1}b)^{n_2}h_3,\ \text{for some}\ h_3 \in K.
	\end{eqnarray*}
    Hence, since $c^{-1}a$ and $b^{-1}c$ are in $K$, we have $w \in K$ so that $H_1 \subseteq K$.
    
    For reverse inequality, first we observe that $a, a^2, a^3 \notin K$. If $a, a^3 \in K$,  then $(a, e, e), (a^3, e, e) \in G_3$, which are not possible by Remark \ref{exp_com_len}. Further, if $a^2 \in K$, then $h = (a^2, e, e) \in \text{St}_{G_3}(\widehat{1})$. However, by Remark \ref{exp_com_len}, we get $|h|_A = 1$, which is a contradiction to Remark \ref{len_even_st}.
    
    Now, let $w \notin H_1$ and $|w|_{a} = m_{1}, |w|_{b} = m_{2}, |w|_{c} = m_{3}$. Thus, $m_1 + m_2 + m_3 = s + 4t$, for some $t \in \mathds{Z}$ and $0 < s \le 3$. If possible, suppose $w \in K$. Then
    \begin{eqnarray*}
    	w &=& a^{m_{1}}b^{m_{2}}c^{m_{3}}g_1,\ \text{for some}\ g_1 \in G_3'\\
    	&=& a^{m_{1}}b^{m_{2} + m_{3}}g_2,\ \text{where }\ g_2 = (c^{-m_{3}}b^{m_{3}})^{-1}g_1 \in K\\
    	&=& a^{m_{1} + m_{2} + m_{3}}g_3, \ \text{where }\ g_3 = (b^{-(m_{2} + m_{3})}a^{m_{2} + m_{3}})^{-1}g_2 \in K\\
    	&=& a^sg_4,\ \text{where}\ g_4 = a^{4t}g_3 \in K.
    \end{eqnarray*}
   Hence, $a^s \in K$, which is a contradiction. Therefore, $K \subseteq H_1$ so that $\text{Rist}_{G_3}(1) = H_1 \times \{e\} \times \{e\}$.
	
	For $k \ge 1$, suppose $\textup{Rist}_{G_3}(\widehat{k}) = \stackrel{3^k}{H_k \times \dots \times H_k}$. Accordingly, $\textup{Rist}_{G_3}(\underset{\text{length } k}{11\ldots 1}) = H_k \times \{e\} \stackrel{3^k-1}{\times \dots \times }\{e\}$.
    We show that $\textup{Rist}_{G_3}(\underset{k+1}{11\ldots 1}) = H_{k+1} \times \{e\} \stackrel{3^{k+1}-1}{\times \dots \times} \{e\}$. 
    
    Let $g\in \textup{Rist}_{G_3}(\underset{\text{length }k+1}{11\ldots 1})$. Clearly, $g \in \textup{Rist}_{G_3}(\underset{\text{length }k}{11\ldots 1})$. Then there exist $g', g'' \in G_3$ such that $\psi_{k+1}(g) = (g', e, \ldots, e)$ and $\psi_{k}(g) = (g'', e, \ldots, e)$ with $\psi_1(g'') = (g', e, e)$. Hence, by inductive hypothesis, we have $g'' \in H_k$ so that $|g''|_{A} \equiv 0\ (\textup{mod}\ 2^{k+1})$. Since $\psi(g'') = (g', e, e)$, by Remark \ref{exp_com_len_d}(\ref{exp_com_twice}), we have $|g'|_{A} = 2|g''|_{A}$ so that $|g'|_{A} \equiv 0\ (\textup{mod}\ 2^{k+2})$. Hence, $g' \in H_{k+1}$ so that $\textup{Rist}_{G_3}(\underset{\text{length } k+1}{11\ldots 1}) \subseteq H_{k+1} \times \{e\} \stackrel{3^{k+1}-1}{\times \dots \times} \{e\}$. 
    
    Conversely, suppose $h \in H_{k+1}$. Clearly, $h \in H_1$ so that there exists $h' \in \textup{Rist}_{G_3}(1)$ such that $\psi(h') = (h, e, e)$. Since $|h|_{A} \equiv 0\ (\textup{mod}\ 2^{k+2})$, again by Remark \ref{exp_com_len_d}(\ref{exp_com_twice}), $|h'|_{A} \equiv 0\ (\textup{mod}\ 2^{k+1})$, i.e., $h' \in H_k$. Thus, by inductive hypothesis, there exists $g \in \textup{Rist}_{G_3}(\underset{\text{length } k}{11\ldots 1})$ such that $\psi_k(g) = (h', e, \ldots, e)$ and hence $\psi_{k+1}(g) = (h, e, \ldots, e)$. Hence,  $\textup{Rist}_{G_3}(\underset{\text{length } k+1}{11\ldots 1}) = H_{k+1} \times \stackrel{3^{k+1}-1}{ \{e\} \times \dots \times \{e\}}$. Consequently, we have $\textup{Rist}_{G_3}(\widehat{k+1}) = \stackrel{3^{k+1}}{H_{k+1} \times \dots \times H_{k+1}}$.
\end{proof}

\begin{theorem}
	The group $G_3$ is a branch group.
\end{theorem}

\begin{proof}
	We have to show that the subgroup $\textup{Rist}_{G_3}(\widehat{k})$ is of finite index in $G_3$, for all $k \in \mathds{N}$. Since $\textup{St}_{G_3}(\widehat{k})$ is of finite index in $G_3$, we show that $\textup{Rist}_{G_3}(\widehat{k})$ is of finite index in $\textup{St}_{G_3}(\widehat{k})$, for all $k \in \mathds{N}$. 
	
	For $g \in \textup{St}_{G_3}(\widehat{k})$, let $\psi_{k}(g) = (g_1, \ldots, g_{3^k})$. By Lemma \ref{ai}, for every $g_i$ there exist $h_i \in H_k$ and $j_i \in \mathds{N}$ with $0 \le j_i  < 2^{k+1}$ such that $g_i = a^{j_i}h_i$. Note that $(h_1, \ldots, h_{3^k}) \in \textup{Rist}_{G_3}(\widehat{k})$. Since $(a^{j_1}h_1, \ldots, a^{j_{3^k}}h_{3^k})$ and $(h_1^{-1}, \ldots, h_{3^k}^{-1})$ are in $\textup{St}_{G_3}(\widehat{k})$, we have $(a^{j_1}, \ldots a^{j_{3^k}}) \in \textup{St}_{G_3}(\widehat{k})$. Hence, for every $g \in \textup{St}_{G_3}(\widehat{k})$, $\psi_{k}(g) = (a^{j_1}, \ldots a^{j_{3^k}})(h_1, \ldots, h_{3^k})$ so that the coset representatives for the quotient group $\textup{St}_{G_3}(\widehat{k})/\textup{Rist}_{G_3}(\widehat{k})$ are from \[\{(a^{i_1}, \ldots, a^{i_{3^k}}) \in \textup{St}_{G_3}(\widehat{k}) : i_1, \ldots, i_{3^k} \in \{0, \ldots, 2^{k+1}-1\}\}.\]  
	Hence, for every $k\in \mathds{N}$, the index of $\textup{Rist}_{G_3}(\widehat{k})$ in $\textup{St}_{G_3}(\widehat{k})$ is bounded by the number ${(2^{k+1})}^{3^k}$. 
\end{proof}

\end{document}